\documentclass[a4paper]{article}
\usepackage{amsfonts}
\usepackage{amssymb}
\usepackage{amsmath}
\usepackage{amsthm}
\usepackage{url}
\usepackage[utf8]{inputenc}
\usepackage[nottoc]{tocbibind}
\usepackage[symbol]{footmisc}
\usepackage{xcolor}
\usepackage{mathrsfs}

\usepackage{dsfont}
\usepackage{hyperref}
\hypersetup{
	colorlinks=true,
	linkcolor=blue,
	filecolor=magenta,      
	urlcolor=cyan,
}
\usepackage{cleveref}
\usepackage{enumitem}
\usepackage{mathtools}

\usepackage{a4wide}

\newtheorem{lemma}{Lemma}[section]
\newtheorem{theorem}[lemma]{Theorem}
\newtheorem*{theorem*}{Theorem}
\newtheorem{definition}[lemma]{Definition}
\newtheorem{proposition}[lemma]{Proposition}

\newtheorem*{proposition*}{Proposition}

\newtheorem{corollary}[lemma]{Corollary}

\newtheorem{notation}[lemma]{Notation}

\newtheorem{result}{Theorem}

\crefalias{result}{theorem}

\newcommand{\ZZ}{\mathbb{Z}}

\newcommand{\NN}{\mathbb{N}}

\newcommand{\inv}{^{-1}}

\newcommand{\abs}[1]{\left|#1\right|}
\newcommand{\norm}[1]{\left|\left|#1\right|\right|}
\newcommand{\set}[1]{\left\{#1\right\}}

\newcommand{\mc}[1]{\mathcal{#1}}

\newcommand{\normal}{\vartriangleleft}

\newcommand{\cl}{\mathrm{Cl}}

\newcommand{\lelangle}{\left\langle}
\newcommand{\rerangle}{\right\rangle}

\newcommand{\hall}[1]{\mc{H}_{#1}}

\title{The spectrum of conjugator length functions}
\author{Lukas Vandeputte\thanks{The author  kindly acknowledges the support by the group of Science Engineering and Technology at KU Leuven Campus Kulak.}}
\begin{document}
	\maketitle
	\begin{abstract}
		A recent program tries to find which functions appear as conjugator length functions. In this note, we show that any (computable) increasing function larger than $n$ appears as $\cl_G$ for some finitely generated (recursively presented) group. On the other hand, we demonstrate that either $\cl_G$ must be constant or $\cl_G(n)\succ n$. Combining these, we obtain a complete description of which functions appear as conjugator length functions of finitely generated groups.
	\end{abstract}
	\section*{Introduction}
	Recently, a lot of effort has been put into classifying which functions appear as conjugator length functions. Given a group $G$ with finite generating set $S$, this function $\cl_{G,S}(n)$ is given by$$
	\max\{\min\{\norm{h}_S\mid h\in G, hg_1h\inv=g_2\}\mid g_1\sim g_2,\norm{g_1}_S+\norm{g_2}_S\leq n\}
	$$
	where $\norm{\_}_S$ denotes the word norm with respect to a generating set $S$. The dependence on $S$ can be suppressed by passing to asymptotics. 
	Most of the effort in classifying which functions appear as conjugator length functions has focused on the class of finitely presented groups. Here the effort has been spearheaded by Bridson and Riley. For a more exhaustive treatment, we refer the reader to \cite{Bridson2026conjugator}, but we give a short rundown. Bridson and Riley successively found that certain classes of nilpotent groups can realise all polynomial functions \cite{Bridson2026Length}, they found a class of groups of conjugator length $n^\alpha$ for $\alpha$ dense in $[2,+\infty)$\cite{Bridson2025snowflake}, found groups with exponential conjugator length and fast-growing conjugator length \cite{Bridson2025groups}.
	More recently, the sharpest result in this fashion was obtained by Gillis \& Wagner in \cite{Gillis2026conjugator}, where they showed that every at least quadratic function that can be obtained as a time function of a so-called $S$-machine can be realised. This set contains all Dehn functions of finitely presented groups, up to standard equivalence.
	A commonality between all these results is that they only obtain functions that are either bounded, linear, or at least quadratic, but they obtain no functions in between.
	
	In the world of finitely generated groups that are not necessarily finitely presented, things have been quieter.
	Of course, all functions mentioned above also remain in this setting. Further, a result from the author \cite{Vandeputte2025residual} shows that every function $\Phi$ satisfying either $\Phi(n)>n$ and $\Phi(3n)>3\Phi(n)$, or satisfying $\Phi(n)\geq n^{1.64+\log_3\log_3(n)}$ appears as the conjugator length function of some finitely generated group. Furthermore, if the function $\Phi$ is computable, then the corresponding group is recursively presented.
	
	In this paper, we strengthen this result further: in \cref{sec:spectrum} we obtain for every (computable) function $f$ which is at least linear, a finitely generated (recursively presented) group $G_f$ such that $\cl_{G, S}(n)\simeq f(n)$.
	In particular, this also answers the question whether $\cl$ can lie strictly between $n$ and $n^2$.
	
	This just leaves the gap between the bounded functions and the linear functions. We demonstrate in \cref{sec:gap} that no functions appear in this gap.
	
	This leads to the following result:	
	
	\begin{result}\label{prop:mainResult}
		Let $f:\NN\rightarrow\NN$ be a non-decreasing function. There exists a finitely generated group $G$ whose conjugator length function satisfies $$\cl_{G}(n)\simeq f(n)$$ if and only if either $f(n)\simeq 1$ or $f(n)\succ n$. Moreover, if $f$ is computable, then $G$ may be chosen to be recursively presented.
	\end{result}
	We end by giving an overview of the proof. 
	Both directions of the proof are treated independently and can be treated as stand-alone.
	The bulk of the paper is spend on the ``$\Leftarrow$'' direction:
	First, for the super-linear functions in \cref{sec:spectrum}, we take two copies of the same group $\mc H_p/K_0$. Conjugacy in these groups is really easily understood. We take a direct product of these groups, where we amalgamate the centre according to the desired function. The conjugators of this new groups still behave like the conjugators in both components, but the geometry is distorted, leading to the ``$\Leftarrow$'' direction of the proof. For the ``$\Rightarrow$'' direction, in \cref{sec:gap} a result from the theory of $FC$ groups allows us to split in the case where eater the centre is finite index, leading to a bounded conjugator length, or where there is an infinite conjugacy class, which we use to demonstrate a linear lower bound on the conjugator length.
	\section{Preliminaries}
	\subsection{Notation and basic definitions}
	We use the convention that $0\in \NN$.
	Let $g_1,g_2\in G$, we denote $[g_1,g_2]=g_1g_2g_1\inv g_2\inv$. We denote $g_1\sim g_2$ if $g_1$ and $g_2$ are conjugate.
	For non-decreasing functions $f,g:\NN\rightarrow\NN$, denote $f(n)\succ g(n)$ if there exists some $C\in \NN$ such that $Cf(Cn+C)+C\geq g(n)$ for all $n\in\NN$. We say $f(n)\simeq g(n)$ if $f\prec g$ and $g\prec f$.\footnote{Note that this equivalence relation is different from the one used for Dehn functions, where one would add a $Cn$ term.}
	
	Let $\norm{\_}_S$ denote the word norm on $G$ where $S$ is some finite generating set.
	It is well known that word norms corresponding to different generating sets are equivalent. That is 
	for any pair $S,S'$, there exists some constant $C$ such that for any $g\in G$,$$
	\frac{1}{C}\norm{g}_S\leq \norm{g}_{S'}\leq C\norm{g}_S.
	$$
	
	Let $g_1,g_2\in G$ be two elements such that $g_1\sim g_2$, denote then$$
	\cl_{G,S}(g_1,g_2)=\min\{\norm{h}_S\mid h\in G, hg_1h\inv=g_2\}
	$$
	and denote 
	$$
	\cl_{G,S}(n)=\max\{\cl_{G,S}(g_1,g_2)\mid g_1,g_2\in G, g_1\sim g_2,\norm{g_1}_S+\norm{g_2}_S\leq n\}.
	$$
	This definition is equivalent to the definition from the introduction by unwinding definitions.
	Let $S_1,S_2$ be two generating sets of the same group $G$, then 
	$\cl_{G,S_1}(n)\simeq \cl_{G,S_2}(n)$. This follows immediately from the equivalence of the norms $\norm{\_}_S$ and $\norm{\_}_{S'}$
	This allows us to suppress the generating set when we are only considering asymptotics.
	\subsection{Recap on $\mc H_p$}
	As a base building block for our groups, we use Hall's groups. These groups were introduced by Hall in \cite{hall1954finiteness}.
	
	We define these groups $\hall p$ through a presentation as$$
	\hall p=\lelangle \{t,a_i,c_j\mid i\in\ZZ,j\in\NN\}\Big\vert \begin{matrix}
		ta_it\inv=a_{i+1}&i\in\ZZ\\
		[t,c_j]=[a_i,c_j]=1&j\in\NN,i\in\ZZ\\
		a_i^p=c_j^p=1&i\in\ZZ,j\in\NN\\
		[a_i,a_{i+j}]=c_j&i\in\ZZ,j\in \NN
	\end{matrix}\rerangle.
	$$
	The following facts are immediate from the presentation, or are alternatively mentioned in \cite[sec. 3.2.]{hall1954finiteness}:
	The group $\mc H_p$ is finitely generated by $S=\{t,a_0
	\}$.
	The elements $\{a_i\mid i\in\ZZ\}$ generate a $2$-step nilpotent group which we call $N_p$. The derived subgroup of $N_p$, which we call $C_p$, is a $\ZZ_p$-module, freely generated by the elements $\{c_j\mid j\in\NN_{>0}\}$. The quotient $N_p/C_p$ is once again a $\ZZ_p$-module, freely generated by the elements $\{a_iC_p\mid i\in\ZZ\}$.
	The quotient $\hall p/N_p$ is an infinite cyclic group, generated by $tN_p$. The group $C_p$ is central in $\hall p$.
	Conjugation by $t$ acts on generators $a_i$ by index shift, and leaves $c_i$ generators invariant.
	It is clear from the presentation that the quotient $\mc H_p/C_p$ is isomorphic to the lamplighter group $\ZZ_p\wr\ZZ$ where $a_0$ generates the lamp group and $t$ generates the street group.

	We have quite a lot of control over the geometry of these groups.
	
	The following statements appeared in \cite{Vandeputte2025residual} for the $p=0$ case. The proofs, however, are identical.

	\begin{definition}
		Let $S_n\subset N_p$ be defined as$$
		S_{n}=\set{a_i\mid i\in [-n,n]}.
		$$
	\end{definition}
	None of the sets $S_n$ generate $N_p$. Nevertheless, we have the following.
	\begin{lemma}{\cite[Lemma 3.5.]{Vandeputte2025residual}}
		\label{prop:support}
		Let $g\in N_p$ and $n\in\NN$. Suppose that $\norm{g}_S\leq n$, then $g\in \langle S_{n}\rangle$.
	\end{lemma}
	We also have a similar result for the derived subgroup:
	\begin{definition}
		Let $S_{c,n}\subset C_p$ be defined as$$
		S_{c,n}=\set{c_i\mid i\in [1,n]}.
		$$
	\end{definition}
	\begin{corollary}{\cite[Corollary 3.8.]{Vandeputte2025residual}}\label{prop:supportC}
		Let $g\in \langle \set{c_i\mid i\in \ZZ}\rangle$ and suppose that $\norm{g}\leq n$, then $g\in \lelangle S_{c,n}\rerangle$.
	\end{corollary}
	
	We also recall a result on conjugacy classes in $\mc H_p$:
	\begin{lemma}{\cite[Corollary 3.11.]{Vandeputte2025residual}}\label{prop:conjInSubgroup}
		Let $g_1,g_2\in \mc H_p$ such that $g_1C_p=g_2C_p$. If $g_1,g_2\notin N_p$, then $g_1\sim g_2$ if and only if $g_1=g_2$. On the other hand if $g_1,g_2\in N_p$, then $g_1\sim g_2$ if and only if $g_2g_1\inv\in [g_1,N_p]$.
	\end{lemma}
	
	\begin{corollary}\label{prop:conjInSubgroupQuotient}
		Let $K\normal C_p$, Let $g_1,g_2\in \mc H_p$ such that $g_1C_p=g_2C_p$. If $g_1,g_2\notin N_p$, then $g_1K\sim g_2K$ if and only if $g_1K=g_2K$.
	\end{corollary}
	\begin{proof}
		The ``$\Leftarrow$'' direction is trivial. Assume $g_1K\sim g_2K$.
		Let $g_3$ be such that $g_1\sim g_3$ and such that $g_2K=g_3K$. By \Cref{prop:conjInSubgroup}, we have $g_1=g_3$, from which the result follows.
	\end{proof}

	\section{Realising the functions}\label{sec:spectrum}
	Our groups will be obtained by taking the direct product of two Hall $p$-groups and amalgamating the two copies of $C_p$ in the right manner.

	Our construction works for strictly increasing functions $f$ such that $f(n)\geq n$. This condition is not a weakening of the conditions of \Cref{prop:mainResult} by the following lemma.
	\begin{lemma}\label{prop:goodRepresentative}
		Let $f:\NN\rightarrow\NN$ be a non-decreasing function such that $f(n)\succ n$, then there exists some strictly increasing function $g:\NN\rightarrow\NN$ such that $g(n)>n$ and such that $f\simeq g$.
		Furthermore, if $f$ is computable, then $g$ can be chosen to be computable.
	\end{lemma}
	\begin{proof}
		This works for $g(n)=f(n)+n+1$. If $f$ is computable, then so is $g$. As $f$ is non-decreasing, and $n$ is strictly increasing, it follows that $g$ is strictly increasing. We have $g(n)>n$ as $f(n)$ is non-negative. As $g$ is strictly larger than $f$, we have $g(n)\succ f(n)$p. Finally, as $f(n)\succ n$, there exists some constant $C\geq 1$ such that $Cf(Cn+C)+C>n$ for any choice of $n$. It follows that $(C+1)f((C+1)n+C+1)+C+1\geq Cf(Cn+C)+C+f(n)>n+f(n)$ and thus $f(n)\succ g(n)$.
	\end{proof}
	\newcommand{\rceilq}{\rceil_{q^\NN}}
	\begin{definition}
		Let $f:\NN\rightarrow\NN$ be a function such that $f(n)> n$. Let $q$ be an integer, denote then $f_q:\NN\rightarrow q^\NN$ which maps $k$ to $\lceil f(q^{k})\rceilq$ where $\lceil\_\rceilq$ denotes rounding up to the nearest power of $q$. For notational convenience we also define $f_q(-1)=-1$\footnote{This in particular guarantees that $f_q(0)\neq f_q(-1)$ later.}.
	\end{definition}
	\begin{lemma}\label{prop:qRounding}
		Let $f$ be non-decreasing, then $f(n)\simeq f_q(\lceil\log_q n\rceil)$. 
	\end{lemma}
	\begin{proof}
		The result is immediate, as $\lceil \_\rceilq$ increases $n$ by at most a factor $q$.
	\end{proof}
	Henceforth, fix a prime $p$ and fix $q$ an integer such that $q>2p+6$.
	\begin{definition}\label{def:theMainGroup}
		Let $f:\NN\rightarrow\NN$ be strictly increasing such that $f(n)> n$, denote $K_f\normal\hall p\times\hall p$ the central subgroup $$
		\lelangle
		\begin{matrix}
			(c_i,1)&\text{ $i$ is \emph{not} a power of $q$}\\
			(1,c_i)&\text{$i$ is \emph{not} of the form $f_q(k)$ for some $k\in\NN$}\\
			(c_{q^k}\inv,c_{f_q(k)})&\text{$k\in\NN$ and $f_q(k)\neq f_q(k-1)$}\\
			(c_{q^k},1)&f_q(k)=f_q(k-1)
		\end{matrix}
		\rerangle
		$$
		and define $G_f$ as $(\hall p\times\hall p)/K_f$.
	\end{definition}
	As $\mc H_p$ is finitely generated, so is $G_f$.
	If $f$ is computable, then $G_f$ is also recursively presented:	
	\begin{lemma}\label{prop:rPresented}
		Let $f$ be a computable function. Then $G_f$ is recursively presented.
	\end{lemma}
	\begin{proof}
		It is clear that $\mc H_p\times \mc H_p$ is recursively presented. It thus suffices to show that $K_f$ is generated by a recursively enumerable set. The set $\{(c_i,1)\mid i\notin q^\NN\}$ is clearly recursive and as $f$ is computable, so is $f_q$. 
		This implies that the set $\{k\in\NN\mid f_q(k)\neq f_q(k-1)\}$ is computable.
		From this, it follows that both the set 
		$\{(c_{q^k},1)\mid f_q(k)=f_q(k-1)\}$ and the set $$\{(c_{q^k}\inv,c_{f_q(k)})\mid\text{$k\in\NN$ and $f_q(k)\neq f_q(k-1)$}\}$$ are computable.
		Furthermore, as $f_q$ is non-decreasing, we may enumerate all elements of the form $f_q(k)$ in order. 
		In particular, for any $i$ we may check if $i$ is of the form $f_q(k)$ by enumerating these elements in order until either $f_q(k)=i$, or until $f_q(k)>i$, in which case $i$ can not be of the form $f_q(k)$. With this it follows that also the set $\{(1,c_i)\mid \text{$i$ is \emph{not} of the form $f_q(k)$ for some $k\in\NN$}\}$ is recursive. As a finite union of recursive sets is recursive, it follows that $K_f$ is recursively-generated. The result now follows as the quotient of a recursively-presented group by a recursively-generated kernel is again recursively-presented.
	\end{proof}
	
	\begin{notation}
		To avoid notational clutter, we denote $t',a_i',c_i'$ for $(t,1)K_f,(a_i,1)K_f$ and $(c_i,1)K_f$ and we denote $t'',a_i'',c_i''$ for $(1,t)K_f,(1,a_i)K_f$ and $(1,c_i)K_f$.
		We also denote $G'$ for the subgroup generated by $t',a_i'$ and denote $G''$ for the subgroup generated by $t'',a_i''$. Both of these are normal in $G_f$.
		Denote $K'$ the kernel of the morphism $\mc H_p\rightarrow G_f$ where we map $t,a_i,c_i$ to $t', a_i',c_i'$. Furthermore, denote $\mc H_p',N_p',C_p'$ and $S'$ the image of $\mc H_p,N_p,C_p$ and $S$, respectively, under that map.
		
		Similarly we define $K'',\mc H_p'',N_p'',C_p''$ and $S''$ using the morphism $\mc H_p\rightarrow G_f$ where we map $t,a_i,c_i$ to $t'', a_i'',c_i''$.
	\end{notation}
	Throughout our arguments, it will be important to have an understanding of $K'$ and $K''$.
	\begin{lemma}
		$$K'=\lelangle\left\{\begin{matrix}
			c_i&\text{ $i$ is \emph{not} a power of $q$}\\
			c_{q^k}&f_q(k)=f_q(k-1)
		\end{matrix}\right\}\rerangle,$$\\ $$K''=\lelangle\{c_i\mid\text{$i$ is \emph{not} of the form $f_q(k)$ for some $k\in\NN$}\\\}\rerangle,$$ and $C_p'=C_p''$.
	\end{lemma}
	\begin{proof}
		The first copy of $C_p$ decomposes as$$V_1\oplus V_2=
		\lelangle\left\{\begin{matrix}
			(c_i,1)&\text{ $i$ is \emph{not} a power of $q$}\\
			(c_{q^k},1)&f_q(k)=f_q(k-1)
		\end{matrix}\right\}\rerangle\oplus 
		\lelangle
		c_{q^k}\mid f_q(k)\neq f_q(k-1)
		\rerangle
		$$
		and the second copy of $C_p$ decomposes as\begin{align*}V_3\oplus V_4=
		&\lelangle\{(1,c_i)\mid\text{$i$ is \emph{not} of the form $f_q(k)$ for some $k\in\NN$}\}\rerangle\\\oplus& \lelangle\{(1,c_i)\mid\text{$i$ is of the form $f_q(k)$ for some $k\in\NN$}\}\rerangle
		\end{align*}
		The generators of $K_f$ of the first and fourth kind are precisely the basis for $V_1$, and the generators of the second kind are precisely the basis for $V_3$. The third kind of relators induces a bijection between the basis of $V_2$ and $V_4$. In particular, this bijection also induces an isomorphism $\varphi: V_2\rightarrow V_4$. From these observations, it follows that $K_f$ decomposes as $V_1\oplus V_3\oplus\{(v\inv,\varphi(v))\mid v\in V_2\}$. This third component intersects both $V_2$ and $V_4$ trivially. It follows that the restricted quotient $C_p\times\{1\}\rightarrow C_p\times\{1\}/K_f$ has kernel $V_1$ and that the restricted quotient $\{1\}\times C_p\rightarrow \{1\}\times C_p/K_f$ has kernel $V_3$. Furthermore, their images are both equal as identified by $\varphi$. We thus have $C_p'=C_p''$. 
		
	\end{proof}
		We will swap between $C_p', C_p''$ and $C_f$ in order to accentuate that we are regarding it as a subgroup of $\mc H_p',\mc H_p''$ or to accentuate neither.
	\begin{definition}
		Let $K_0\normal \hall p$ be the central subgroup generated by $\{c_i\mid \text{$i$ is not a power of $q$}\}$.
	\end{definition}
	Notice that as $f_q(k)$ is always a power of $q$, we have that $K_0$ is contained in both $K'$ and $K''$.
	Swapping between $C_p'$ and $C_p''$ can be easily understood in a geometric sense.
	\begin{lemma}\label{prop:Cconversion}
		For any $k\in\NN$ we have
		$$\{1,c_{q^i}'\mid i\leq k\}=\{1,c_i'\mid i\leq  q^k\}=\{1,c_i''\mid i\leq f_q(k)\}=\{1,c_{q^i}''\mid q^i\leq f_q(k)\}.$$
	\end{lemma}
	\begin{proof}
		The set $K'$ contains $K_0$. It thus follows that if $i$ is not a power of $q$, then $c_i'=1$. From this, the first equality follows. As $K''$ also contains $K_0$, the third equality also follows.
		
		Suppose now that $c_{q^i}'$ is non trivial and $i\leq k$. Non triviality of $c_{q^i}'$ guarantees that $f_q(i)\neq f_q(i-1)$, as otherwise, $c_{q^i}'$ would have been killed by one of the relators of $K_f$ of the fourth kind.
		In this case we have by some relator of $K_f$ the third kind that $c_{q^i}'=c_{f_q(i)}''$. Furthermore, as $q^i\leq q^k$ and as $f_q$ is non-decreasing, it follows that $f_q(i)\leq f_q(k)$ and thus that $c_{f_q(i)}''$ is contained in the right set.
		
		On the other hand, suppose $c_{q^i}''$ is non trivial and $q^i\leq f_q(k)$. By the relators of $K''$, it follows that $q^i$ is of the form $f_q(j)$ for some $j$. We may choose $j$ minimal such that $f_q(j)=q^i$. In this case, as $f_q$ is non decreasing, we have that $j\leq k$. Furthermore, by choice of $j$, we have that $f_q(j)\neq f_q(j-1)$. Using the third set of relators of $K_f$, we have that $c_{q^j}'=c_{q^i}''$. It follows that $c_{q^i}''$ is contained in the left set.
	\end{proof}
	The quotient maps $\mc H_p\rightarrow\mc H_p/K_0$, $\mc H_p\rightarrow\mc H_p/K'$ and $\mc H_p\rightarrow\mc H_p/K''$ do not identify more than expected.
	\begin{lemma}\label{prop:intersection}
		Let $S_1,S_2\subset \NN$ such that $S_1\cap S_2=\emptyset$, then $\langle c_iK_0\mid i\in S_1\rangle\cap\langle c_iK_0\mid i\in S_2 \rangle<\mc H_p/K_0$ is trivial. The same holds if we replace $c_iK_0$ with $c_i'$ or with $c_i''$.
	\end{lemma}
	\begin{proof}
		The subgroup $C_p$ decomposes as $\langle\{c_i\mid \text{$i$ is a power of $q$}\}\rangle\oplus\langle\{c_i\mid \text{$i$ is not a power of $q$}\}\rangle$. The quotient map $\mc H_p\rightarrow\mc H_p/K_0$ is injective on the first component and kills the second component. In particular, the quotient map is equivalent to projection on the first component.
		The subgroups $\langle c_i\mid i\in S_1\rangle$ and $\langle c_i\mid i\in S_2\rangle$ also decompose as $$
		\langle c_i\mid i\in S_1,\text{$i$ is a power of $q$}\rangle\oplus \langle c_i\mid i\in S_1,\text{$i$ is not a power of $q$}\rangle
		$$
		and 
		$$
		\langle c_i\mid i\in S_2,\text{$i$ is a power of $q$}\rangle\oplus \langle c_i\mid i\in S_2,\text{$i$ is not a power of $q$}\rangle.
		$$
		The projections on the first component are thus given by $$\langle c_i\mid i\in S_1,\text{$i$ is a power of $q$}\rangle$$
		and $$
		\langle c_i\mid i\in S_2,\text{$i$ is a power of $q$}\rangle.$$
		These subgroups have trivial intersection as $C_p$ is freely generated by $\{c_i\mid i\in \NN_{>0}\}$ and as $S_1$ and $S_2$ are disjoint.
		
		The same argument works for the other cases, where we replace ``a power of $q$'' with ``of the form $q^k$ with $f_q(k)\neq f_q(k-1)$'' or ``of the form $f_q(k)$''.
	\end{proof}
	
	The geometry of $G_f$ is closely related to the geometry of $\mc H_p'$ and $\mc H_p''$.
	\begin{lemma}\label{prop:decomp}
		Let $g\in G_f$, there exist $g'\in \mc H_p'$ and $g''\in \mc H_p''$ such that $g=g'g''$ and such that $\norm{g}_{S'\cup S''}=\norm{g'}_{S'}+\norm{g''}_{S''}$.
	\end{lemma}
	\begin{proof}
		We may write $g$ as a product $s_1s_2 \cdots s_n$ where $s_i\in S'\cup S''$. As the elements from $S'$ commute with the elements from $S''$, we may reorder the generators such that the first $k$ are all in $S'$ and the other $n-k$ generators are all in $S''$. In this case let $g'=s_1\cdots s_k$ and let $g''=s_{k+1}\cdots s_n$. Clearly $g=g'g''$. Now by construction, we have $\norm{g'}_{S'}+\norm{g''}_{S''}\leq \norm{g}_{S'\cup S''}$ and by the triangle inequality, $\norm{g}_{S'\cup S''}\leq \norm{g'}_{S'\cup S''}+\norm{g''}_{S'\cup S''}\leq\norm{g'}_{S'}+\norm{g''}_{S''}$.
	
	\end{proof}


	With our current understanding of $G_f$, it is not difficult to bound $\cl_{G_f}$ from below.
	\begin{proposition}\label{prop:clLower}
		$\cl_{G_f}(n)\succ f(n)$
	\end{proposition}
	\begin{proof}
		Let $n$ be arbitrary. Let $k$ be minimal such that $q^k\geq n$ and let $k_0$ be minimal such that $f_q(k)=f_q(k_0)$.
		Consider $g_1=a_0''$ and $g_2=a_0''c_{f_q(k)}''=a_0''c_{q^{k_0}}'$. The sum of the word norms of these elements is at most $6+4q^{k_0}\leq 6+4qn$ with respect to the generating set $S'\cup S''$.
		Let $h$ be an element of $G_f$ such that $hg_1h\inv=g_2$. Let $h=h'h''$ as in \Cref{prop:decomp}. Notice that $h'$ commutes with $h''$ and with $g_1$. In particular, we have $h''g_1(h'')\inv=g_2$ where $\norm{h''}_{S''}\leq \norm{h}_{S'\cup S''}$. It follows that $\cl_{G_f,S'\cup S''}(g_1,g_2)$ is bounded from below by $\cl_{\mc H_p'',S''}(g_1,g_2)$.
		Let $h$ be any element in $\mc H_p''$ of word norm at most $\frac{f_q(k)}{2}-2$ with respect to $S''$.
		In this case, by the triangle inequality, $g_2g_1\inv=hg_1h\inv g_1\inv$ is of word norm at most $f_q(k)-2$, with respect to $S''$. However, by \Cref{prop:intersection}, we have that $g_2g_1\inv=c_{f_q(k)}''$ is not contained in $\langle \{c_i''\mid i\leq f_q(k)-2 \rangle$.
		It follows by \Cref{prop:supportC} that $g_2g_1\inv$ is of word norm at least $f_q(k)$ in $\mc H_p''$. It follows for any $n$ that $\cl_{G_f,S'\cup S''}(6+4qn)\geq \cl_{\mc H'',S''}(6+4qn)\geq \frac{f_q(k)}{2}-2$ where $k=\log_q\lceil n\rceilq$.
		The result now follows as by \Cref{prop:qRounding}, $f_q(\log_q\lceil n\rceilq)\succ f(n)$.
	\end{proof}
	To establish an upper bound on $\cl_{G_f}$, a first step is understanding conjugators in the ``components'' $\mc H_0/K'$ and $\mc H_0/K''$.
	\begin{proposition}\label{prop:HallConjugator}
		Let $g_1,g_2\in \mc H_p/K_0$ be such that $g_1C_p=g_2C_p$. Suppose that $g_1\in N_p\backslash C_p$. Then $g_1\sim g_2$. Furthermore if $g_2g_1\inv \in \langle \{c_{q^{i}}\mid i\leq N\}\rangle$, and $\norm{g_1}\leq n$, then $\cl_{\mc H_p/K_0,S}(g_1,g_2)\leq q(p+2)(\norm{g_1}_S+q^{N})$. The same result holds when we replace $K_0$ with $K'$ or $K''$.
	\end{proposition}
	\begin{proof}
		We provide the proof for $K_0$, but the proof for $K'$ and $K''$ is verbatim identical after replacing all mentions of $K_0$ with $K'$ or $K''$. 
		All computations in this proof are carried out in $\mc H_p/K_0$ unless otherwise stated.
		In this proof, we will first construct a sequence of elements $h_m$ that come closer and closer to conjugating $g_1$ into $g_2$. We will then use some length estimates to demonstrate that for the correct choice of $m$, $h_m$ must in fact be a conjugator, which will then lead to the result.
		
		We may write $g_1C_p$ as $\prod a_i^{\alpha_i}C_p$ for some $\alpha_i\in \ZZ_p$. By \Cref{prop:support}, $\alpha_i$ may only be non-$0$ for $\abs{i}\leq n$. Denote $S(g_1)=\{i\in\ZZ\mid \alpha_i\not\equiv 0\mod p\}$ and denote $I=\max S(g_1)$. We first show by induction on $m$ that there exists some $h_m$ of the form $$\prod_{k=0}^{m} a_{I+q^{k}}^{\beta_k}$$
		such that $h_mg_1h_m\inv g_2\inv\in\langle\{c_{q^i}\mid i>m\}\rangle$.
		
		Notice that such an element $h_m$ has a word norm at most $\sum_{k=0}^m \abs{\tilde\beta_k}+ \abs{I}+q^m$ where $\tilde \beta_k\in\ZZ$ is any representative of $\beta_k\in\ZZ_p$.
		As $h_m$ only depends on the residue class mod $p$, we can bound this word norm (very crudely) from above by $(p+1)q^m+\abs I$.

		The induction statement is trivially true for $m=-1$, in which case the product runs over the empty set and we obtain that $h_{-1}$ is the trivial element.
		Suppose by induction that such $h_{m-1}$ exists.
		In this case we can write $[h_{m-1},g_1]g_1g_2\inv$ as $$
		\prod_{i=m}^{\infty}c_{q^k}^{\gamma_k}
		$$
for some $\gamma_k\in\ZZ_p$.\footnote{For now we refrain from giving a bound on the index set, but only finitely many of the exponents $\gamma_k$ can be non-$0$, making the actual index set of the product finite.}		
		
		Consider the commutator $[a_{I+q^m}^\beta,g_1]$. Using bilinearity of commutators in $2$-step nilpotent groups, we may rewrite this as $$
		\prod_{j\in S(g_1)} c_{I+q^m-j}^{-\beta\alpha_j}.
		$$
		Note that $I+q^m-j$ only takes values at least $q^m$.
		Denote $h_{m}=a_{I+q^m}^{\gamma_m/\alpha_I}h_{m-1}$ where the division $\gamma_m/\alpha_I$ is carried out mod $p$. First, notice that $h_m$ is of the desired form. Furthermore, $h_mg_1h_m\inv g_2\inv$ can be rewritten as $[h_m,g_1]g_1g_2\inv$ which is by bilinearity given by $[a_{I+q^m}^{\gamma_m/\alpha_I},g_1][h_{m-1},g_1]g_1g_2\inv$ or thus
		$$
		\prod_{\substack{j\in S(g_1)\\ I+q^m-j\text{is a power of $q$}}} c_{I+q^m-j}^{-\gamma_m/\alpha_I\alpha_j}\prod_{k=m}^{\infty}c_{q^k}^{\gamma_k}.
		$$
		The element $c_{q^m}$ appears on the left precisely when $j=I$, in which case it has exponent $-\gamma_m\mod p$. On the right-hand side, $c_{q^m}$ appears for $k=m$, where it appears with exponent $\gamma_m$.
		Multiplying together, the exponent of $c_{q^m}$ is given by a multiple of $p$, which implies that the exponent vanishes.
		For any $i<q^m$, $c_i$ appears neither on the left-hand side nor the right-hand side. Finally, as $c_i$ vanishes whenever $i$ is not a power of $q$, we have that $h_m$ satisfies the induction hypothesis.
		
		Now consider $h=h_m$ where $m$ is minimal such that $q^m>\max\{\norm{g_1}, q^N\}$.
		We may bound the word norm of $[h_m,g_1]$ by $2(p+1)q^m+2\abs I+2q^m\leq2(p+1)q^m+4q^m$.
		
		Remember that, as specified above \Cref{def:theMainGroup}, $2(p+1)q^m+4\norm{g_1}< q^{m+1}$.
		It follows by \Cref{prop:supportC} that $[h_m,g_1]$ is contained in $\langle\{ c_i\mid i<q^m+1\}\rangle$
		which is a subgroup of $\lelangle\{ c_i\mid i<q^{m+1}\}\rerangle$.
		Similarly, $g_2g_1\inv$ is contained in $$
		\langle\{ c_i\mid i<q^N\}\rangle\subset \langle\{ c_i\mid i<q^{m+1}\}\rangle.
		$$
		As $c_i$ vanishes whenever $i$ is not a power of $q$, the above subgroup is furthermore equal to $$
		\langle\{ c_i\mid i\leq q^m\}\rangle.
		$$
		By \Cref{prop:intersection}, however, this subgroup has trivial intersection with $$
		\langle\{c_{i}\mid i>q^m\}\rangle.
		$$
		It follows that $[h_m,g_1]g_1g_2\inv$, which lies in both, must be trivial and thus that $h_mg_1h_m\inv=g_2$.
		The result follows as the norm of $h_m$ is at most $(p+1)(q^m+I)$ where $q^m$ is bounded by $q\max\{\norm{g_1},q^N\}$ and where $I$ is bounded by $\norm{g_1}$.
	\end{proof}
	\Cref{prop:HallConjugator} only works for certain elements of $\mc H_p/K_0$. This, however, is not a restriction.
	\begin{lemma}\label{prop:whenConjugate}
		Let $g_1,g_2\in G_f$ and let $g_1=g_1'g_1'',g_2=g_2'g_2''$ as in \Cref{prop:decomp}. Suppose that $g_1C_f=g_2C_f$. Then\begin{itemize}
			\item if $g_1'\in N_p'\backslash C_p'$ or $g_1''\in N_p''\backslash C_p''$, then $g_1\sim g_2$;\\
			\item otherwise $g_1\sim g_2$ if and only if $g_1=g_2$.
		\end{itemize}
	\end{lemma}
	\begin{proof}
		For the first, without loss of generality, assume $g_1'\in N_p'\backslash C_p'$. As $C_p'=C_p''$, we have that $g_2'(g_2''(g_1'')\inv)\in \mc H_p'$. We may thus assume that $g_2''=g_1''$ by replacing $g_2'$ by $g_2'(g_2''(g_1'')\inv)$ and $g_2''$ by $g_1''$. 
		
		In this case, by \Cref{prop:HallConjugator} there exists some $h'\in \mc H'_p$ such that $h'g_1'(h')\inv=g_2'$. As $h'$ commutes with $g_1''$, it follows that $h' g_1'g_1'' (h')\inv= g_2'g_2''$.
		
		For the second case, we have $g_1'\in C_p'$ or $g_1'\notin N_p'$, and we have $g_1''\in C_p''$ or $g_1''\notin N_p''$.
		Let now $h=h'h''$ be such that $hg_1h\inv=g_2$. Suppose that $g_1'\in C_p'$, then clearly $h'g_1'(h')\inv=g_1'$. Now assume $g_1'\notin N_p'$. As $g_1C_f=g_2C_f$, we have $g_1'C_p'=g_2'C_p'$. By \Cref{prop:conjInSubgroupQuotient}, we once again have $h'g_1'(h')\inv=g_1'$. By the same arguments, $h''g_1''(h'')\inv=g_1''$.
		It follows that $hg_1h\inv=g_1$, as had to be shown.
	\end{proof}
	This finally allows us to establish an upper bound on $\cl_{G_f}(n)$.
	\begin{proposition}\label{prop:clUpper}
		$\cl_{G_f}(n)\prec f(n)$
	\end{proposition}
	\begin{proof}
		Let $g_1,g_2\in G_f$ be arbitrary such that $\norm{g_1}_S+\norm{g_2}_S\leq n$ and suppose that $g_1$ and $g_2$ are conjugate. As in \Cref{prop:decomp}, let $g_1=g_1'g_1''$ and let $g_2=g_2'g_2''$. In particular, $ g_1'C_p'$ must be conjugate to $g_2'C_p'$ and $ g_1''C_p''$ must be conjugate to $g_2''C_p''$. Let $h'$ and $h''$ be conjugators such that $h'g_1'(h')\inv C_p'=g_2' C_p'$ and  $h''g_1''(h'')\inv C_p''=g_2'' C_p''$. As $\mc H_p/C_p\simeq \ZZ_p\wr\ZZ$, we can use the linear bound on $\cl_{\ZZ_p\wr\ZZ}$ obtained in \cite{sale2016conjugacy}, which allows us to pick $h'$ and $h''$ such that $\norm{h'}\leq C(\norm{g_1'}+\norm{g_2'})$ and $\norm{h''}\leq C(\norm{g_1''}+\norm{g_2''})$ for some uniform constant $C$.
		
		We may now replace $g_1$ with $h'h'' g_1(h'h'')\inv$ at the cost of adding a linear term to $\cl_{G_f}(n)$ and multiplying $n$ by a constant factor. 
		By \Cref{prop:clLower} we know that $\cl_{G_f}\succ n$ so this procedure does not change the $\simeq$-equivalence class of $\cl_{G_f}(n)$ and allows us to assume that $g_1'C_p'=g_2'C_p'$ and $g_1''C_p''=g_2''C_p''$.
		
		In this case, we may, by \Cref{prop:whenConjugate}, assume that $g_1'\in N_p'\backslash C_p'$ or $g_1''\in N_p''\backslash C_p''$.
		
		Denote $c=g_2 g_1\inv=g_2'(g_1')\inv g_2''(g_1'')\inv\in C_f$, denote $g_2'(g_1')\inv=c'$ and $g_2''(g_1'')\inv=c''$. By \Cref{prop:supportC}, we have that $c'$ and $c''$ lie in $\langle\{c_i'\mid i<n\}\rangle\subset C_p'$ and $\langle\{c_i''\mid i<n\}\rangle\subset C_p''$ respectively.
		
		By \Cref{prop:Cconversion}, the former subgroup is $\langle \{c_{q^k}'\mid q^k<n\}\rangle$ which, again by \Cref{prop:Cconversion}, is equal to $\langle \{c''_{f_q(k)}\mid q^k<n\}\rangle$. The latter subgroup is equal to the smaller subgroup $\langle\{c_{f_q(k)}''\mid f_q(k)<n\}\rangle$.
		
		First assume that $g_1'\in N_p'\backslash C_p'$. We may write $g_2$ in the form $g_1'c\:g_1''$ with $c\in\langle \{c_{q^k}'\mid q^k<n\}\rangle.$ By \Cref{prop:HallConjugator}, there exists some $h'\in \mc H_p'\simeq \mc H_p/K'$ such that $h'g_1'(h')\inv=g_1'c$. Furthermore, this $h'$ can be chosen of norm at most $q(p+2)(n+n)$, which is linear in $n$. As elements of $\mc H_p'$ and $\mc H_p''$ commute in $G_f$, we have $h'\:g_1'g_1''\:(h')\inv =g_2'g_2''$.
		
		On the other hand, suppose $g_1''\in N_p''\backslash C_p''$. we write $g_2$ in the form $ g_1'\:g_1''c$ with $c\in\langle \{c_{f_q(k)}''\mid q^k<n\}\rangle$.
		Again by \Cref{prop:HallConjugator} there exists some $h''\in \mc H_p''\simeq\mc H_p/K''$ such that $h''g_1''(h'')\inv=g_1''c$. In this case, $h''$ can be chosen of norm at most $q(p+2)(n+f_q(\lceil\log_q(n)\rceil))$. This function grows like $f(n)$ by \Cref{prop:qRounding}.
		For the same reason as before, $h''\:g_1'g_1''\:(h'')\inv= g_2'g_2''$.
		
		In both cases, the result follows.
		
	\end{proof}
	\section{The Gap}\label{sec:gap}
	In this section, we will establish the gap between the bounded and linear functions from \Cref{prop:mainResult}.
	The main ingredient we need to demonstrate the existence of the gap is the following statement:
	\begin{theorem}\label{prop:alternative}
		The following are equivalent for a finitely generated group $G$\begin{itemize}
			\item The elements of $G$ have finite conjugacy classes\\
			\item The index $[G:Z(G)]$ is finite\\
		\end{itemize}
	\end{theorem}
	The above statement is quite standard in the theory of FC-groups and is a direct consequence of \cite[\S 2. Theorem 2.]{Baer1948Finiteness}, using the fact that a finite product of finite normal subgroups is a finite normal subgroup.
	\begin{proposition}\label{prop:fiCenter}
		If $Z(G)$ is of finite index in the finitely generated group $G$, then $\cl_G(n)$ is bounded by a constant.
	\end{proposition}
	\begin{proof}
		Let $R$ be a set of representatives of the quotient $G/Z(G)$ and let $g$ be arbitrary; conjugating with an element of $Z(G)$ leaves $g$ invariant. The conjugacy class of $g$ is thus given by $\{rgr\inv\mid r\in R\}$. The conjugator length is thus bounded by $$\max_{r\in R}\{\norm{r}_S\}.$$
	\end{proof}
	\begin{proposition}\label{prop:infiniteConjugacyClass}
		Let $G$ be a finitely generated group with an infinite conjugacy class; then $\cl_G(n)\succ n$.
	\end{proposition}
	\begin{proof}
		Let $g\in G$ have an infinite conjugacy class and fix $S$ a generating set of $G$.
		Let $C_n$ be the set$$
		\{hgh\inv\mid \norm{h}_S\leq n\}.
		$$
		The set $C_n$ is clearly finite. We first show that $C_n$ is always strictly larger than $C_{n-1}$. Indeed, suppose that $C_n=C_{n-1}$. Conjugating an element of $C_{n-1}$ with an element of $S$ or $S\inv$ clearly gives an element of $C_n$. As $C_n=C_{n-1}$ it follows that $C_{n-1}$ is invariant under conjugation with elements of $S$ and $S\inv$. It follows that $C_{n-1}$ is also invariant under conjugation with $\langle S\rangle=G$. In particular, $C_{n-1}$ must contain the conjugacy class of $g$. This gives rise to a contradiction as the class of $g$ is infinite but $C_{n-1}$ is finite.
		
		For any $n$, fix $g_n\in C_n\backslash C_{n-1}$. These elements exist by the previous.
		By definition, we have $\cl(g,g_n)=n$. Furthermore, the norm of $g_n$ is at most $2n+\norm g$. Indeed, let $h$ be a conjugator of length $n$. Then $g_n=hgh\inv$.
		It follows that $\cl_{G,S}(2\norm{g}_S+2n)\geq n$ from which the result follows.
	\end{proof}
	\begin{proof}{\Cref{prop:mainResult}}
		For the $\Leftarrow$ direction, for the bounded functions (which are all equivalent under $\simeq$) we may consider any group with finite-index centre by \Cref{prop:fiCenter}. For the $f(n)\succ n$ case, we may, by \Cref{prop:goodRepresentative}, assume $f$ is strictly increasing and that $f(n)> n$ for any $n\in\NN$. Consider then the group $G_f$. If $f$ is computable, then we know $G_f$ is recursively presented by \Cref{prop:rPresented}. We also know that $\cl_{G_f}(n)\simeq f(n)$ by \Cref{prop:clLower} and \Cref{prop:clUpper}. For the $\Rightarrow$ direction, by \Cref{prop:alternative}, either a group has a finite-index centre, in which case $\cl_{G}(n)$ is bounded by \Cref{prop:fiCenter}, or it has an infinite conjugacy class, in which case $\cl_{G}(n)\succ n$ by \Cref{prop:infiniteConjugacyClass}.
	\end{proof}
	\bibliographystyle{alpha}
	\bibliography{bibliography}
\end{document}